\documentclass{article}
\usepackage{cite}
\usepackage{amsmath,amssymb,amsthm}
\usepackage{titlesec,hyperref}
\usepackage{color}
\usepackage{fancyhdr}
\usepackage[margin=3cm]{geometry}

\newtheorem{theo}{Theorem}[section]
\newtheorem{lemm}[theo]{Lemma}

\numberwithin{equation}{section}
\allowdisplaybreaks
\begin{document}
	\title{On the Cauchy problem for a weakly dissipative Camassa-Holm  equation in critical Besov spaces }
	\author{
		Zhiying $\mbox{Meng}^1$ \footnote{email: mengzhy3@mail2.sysu.edu.cn} \quad and\quad
		Zhaoyang $\mbox{Yin}^{1,2}$ \footnote{email: mcsyzy@mail.sysu.edu.cn}\\
		$^1\mbox{Department}$ of Mathematics, Sun Yat-sen University,\\
		Guangzhou, 510275, China\\
		$^2\mbox{Faculty}$ of Information Technology,\\
		Macau University of Science and Technology, Macau, China
	}
	\date{}
	\maketitle
	\begin{abstract}
		In this paper, we mainly consider the Cauchy problem  of a weakly dissipative Camassa-Holm  equation.  We first establish the local well-posedness of equation in Besov spaces $B^{s}_{p,r}$ with $s>1+\frac 1 p$ and $s=1+\frac 1 p , r=1,~p\in [1,\infty).$ Then, we prove the global existence for small data, and present two blow-up criteria. Finally,
		we  get two blow-up results,  which can be used in the proof of the ill-posedness  in critical Besov spaces.
	\end{abstract}
	
	
	\tableofcontents
	\section{Introduction}
	\par
	\quad\quad In this paper, we consider   the  Cauchy problem for the following   weakly dissipative Camassa-Holm  equation
	\begin{equation}\label{0.1}
		\left\{\begin{array}{l}
			u_t-u_{txx}+3uu_x+\lambda(u-u_{xx})=2uu_x+uu_{xxx}+\alpha u+\beta u^2u_x+\gamma u^3u_x+\Gamma u_{xxx},~x\in\mathbb{R},\ t>0, \\
			u(0,x)=u_0,
		\end{array}\right.
	\end{equation}
	where  $\alpha, \beta,\gamma,\Gamma$ are arbitrary real constants and $\lambda>0.$  This equation was
	proposed by Freire  in \cite{Igor02decay}. Let $h(u)=(\alpha+\Gamma)u+\frac{\beta}{3}u^3+\frac{\gamma}{4}u^4, ~\Lambda ^{-2}=(1-\partial_{xx})^{-1},$ the equation \eqref{0.1} can be rewritten as
	\begin{equation}\label{001}
		\left\{\begin{array}{l}
		u_t+(u+\Gamma)u_x+\lambda u=Q,~x\in\mathbb{R},\ t>0, \\
		u(0,x)=u_0,
\end{array}\right.
\end{equation}
where $Q=\Lambda ^{-2}\partial_x\Big(h(u)-u^2-\frac{1}{2}u_x^2\Big).$ 

If $\lambda=\beta=\Gamma=0$ and $\alpha\neq 0,$ it reduces to the Dullin-Gottwald-Holm equation \cite{Igor2019bdd,Dullin2001An}.  If $\beta=\gamma=0,~\lambda>0$ and $ \alpha \Gamma\neq 0$, the equation \eqref{001} becomes the weakly dissipative DGH equation \cite{DGH.NE2013,DGH.NE2016,DGH.ZZY2015}.  If $\alpha=\beta=\gamma=\Gamma=0$ and $\lambda>0$, the equation \eqref{001} becomes  weakly  dissipative CH equation \cite{WDCH.wsyy2006,DCH.wsyyz2009}.  If  $\lambda=\alpha=\beta=\gamma=\Gamma=0,$  we deduce that 
	\begin{align}\label{ch}
		(1-\partial_x^2)u_t=3uu_x-2u_xu_{xx}-uu_{xxx},
	\end{align}
	which is the famous Camassa-Holm (CH) equation.
	There are many properties of the CH equation. For example, it is completely integrable \cite{Con.On2001,Constantin2006,Constantin1999} and has a Hamiltonian structure \cite{1981Symplectic}. Its  solitary waves solutions and peakon solutions of the form $ce^{|x-ct|}$ were studied in \cite{Camassa.1994A,Constantin2000Stability,Constantin2000Stability1,YE2021A}. The local well-posedness and ill-posedness global strong solutions, blow-up strong solutions of the CH equation were investigated in  \cite{Guoyy2021Ill,Brandole2014Blowup,Constantin2000,Constantin1998,Constantin1998w-l,Constantin1998w-b,Danchin2001inte,Guozi2019ill,Liuyue2006glo,LiYi,Li2021,Li2016nwpC,Rodr2001CH,Yinzy2004wbglo,YE2021A,Lij022Ill}. The global weak solutions, global  conservative solutions and dissipative solutions f the CH equation also were studied in \cite{Bressan2007gd,Bressan2006g-c,Constantin2000gw,Holden07glorod,Holden07gloCH,Xin2000ws}.
	
	Recently, Freire eatablished the local  well-posedness and blow-up phenomena for the equation \eqref{001} in  Sobolev spaces \cite{Igor02decay}. In \cite{Silva2020inte}, the authors studied the integrability and  global existence of  solution for the equation \eqref{001} under some conditions. Therefore, it is worth noting that the local well-posedness and ill-posedness, global strong solutions, blow-up phenomenan of \eqref{001} in Besov spaces have not yet been studied. Hence, in this paper, we will study these problems. It is worthy to note that the dissipative term $\lambda (u-u_xx)$  in the equation \eqref{001} do have impacts on the global existence, which is proved  below in Theorem \ref{global}.
	
	This article is organized as follows. In Section 2, we introduce some basic properties in Besov spaces and  some prior estimates about the transport equation. In Section 3, we establish the local well-posedness of the Cauchy problem of \eqref{001}.  Section 4  is devoted to discussing the the global strong solution of \eqref{001} with small initial values, and give two blow-up criteria of the equation \eqref{001}. Then, we gain two blow-up results. Finally, we prove that the Cauchy problem of \eqref{001} is ill-posed  in $B^{\frac 3 2}_{2,r}$ with $1<r\leq \infty.$ 
	\section{Preliminaries}
	\par
	
	~~	In this section, we will present some propositions about the Littlewood-Paley decomposition and Besov spaces.
	Now we state some useful results in the transport equation theory, which are important to the proof of our main theorem later. We first give the following equation
	\begin{equation}\label{transport}
		\left\{\begin{array}{l}
			f_t+v\cdot\nabla f=g,~ t>0, \\
			f(0,x)=f_0(x),~ x\in\mathbb{R}.
		\end{array}\right.
	\end{equation}
	\begin{lemm}\cite{Chemin2011}\label{R}
		Let $s,t>0,~p,r\in[1,\infty].$ Define $R_j=[v\cdot \nabla,\Delta_j]f.$ There exists a constants $C$ such that
		\begin{align*}
			\|(2^{js}\|R_j\|_{L^p})_{j}\|_{l^r(\mathbb{Z})}\leq C\Big(\|\nabla v\|_{L^{\infty}}\|f\|_{B^s_{p,r}}+\|\nabla v\|_{B^{s-1+t}_{p,r}}\|\nabla f\|_{B^{-t}_{\infty,\infty}}\Big).
		\end{align*}
		Then, if $f$ solves the equation \eqref{transport}, we have
		\begin{align*}
			\|f(t)\|_{B^s_{p,r}}\leq \|f_0\|_{B^s_{p,r}}+C\int_0^t\Big(\|\nabla v\|_{L^{\infty}}\|f\|_{B^s_{p,r}}+\|\nabla v\|_{B^{s-1+t}_{p,r}}\|\nabla f\|_{B^{-t}_{\infty,\infty}}+\|g\|_{B^s_{p,r}}\Big)(t')dt'.
		\end{align*}
	\end{lemm}
	\begin{lemm}\label{existence}\cite{Chemin2011}
		Let $1\leq p\leq p_1\leq\infty,\ 1\leq r\leq\infty,\ s> -d\min(\frac 1 {p_1}, \frac 1 {p'})$. Let $f_0\in B^s_{p,r}$, $g\in L^1([0,T];B^s_{p,r})$, and let $v$ be a time-dependent vector field such that $v\in L^\rho([0,T];B^{-M}_{\infty,\infty})$ for some $\rho>1$ and $M>0$, and
		$$
		\begin{array}{ll}
			\nabla v\in L^1([0,T];B^{\frac d {p_1}}_{p_1,\infty}), &\ \text{if}\ s<1+\frac d {p_1}, \\
			\nabla v\in L^1([0,T];B^{s-1}_{p,r}), &\ \text{if}\ s>1+\frac d {p_1}\ or\ (s=1+\frac d {p_1}\ and\ r=1).
		\end{array}
		$$
		Then the equation \eqref{transport} has a unique solution $f$ in \\
		-the space $C([0,T];B^s_{p,r})$, if $r<\infty$; \\
		-the space $\Big(\bigcap_{s'<s}C([0,T];B^{s'}_{p,\infty})\Big)\bigcap C_w([0,T];B^s_{p,\infty})$, if $r=\infty$.
	\end{lemm}
	\begin{lemm}\label{priori estimate}\cite{Chemin2011,Li2016nwpC}
		Let $s\in\mathbb{R},\ 1\leq p,r\leq\infty$.
		There exists a constant $C$ such that for all solutions $f\in L^{\infty}([0,T];B^s_{p,r})$ of \eqref{transport} in one dimension with initial data $f_0\in B^s_{p,r}$, and $g\in L^1([0,T];B^s_{p,r})$, we have for a.e. $t\in[0,T]$,
		$$ \|f(t)\|_{B^s_{p,r}}\leq \|f_0\|_{B^s_{p,r}}+\int_0^t \|g(t')\|_{B^s_{p,r}}dt'+\int_0^t V^{'} (t^{'})\|f(t)\|_{B^s_{p,r}}dt{'} $$
		or
		$$ \|f(t)\|_{B^s_{p,r}}\leq e^{CV(t)}\Big(\|f_0\|_{B^s_{p,r}}+\int_0^t e^{-CV(t')}\|g(t')\|_{B^s_{p,r}}dt'\Big) $$
		with
		\begin{equation*}
			V'(t)=\left\{\begin{array}{ll}
				\|\nabla v\|_{B^{s+1}_{p,r}},\ &\text{if}\ s>\max(-\frac 1 2,\frac 1 {p}-1), \\
				\|\nabla v\|_{B^{s}_{p,r}},\ &\text{if}\ s>\frac 1 {p}\ \text{or}\ (s=\frac 1 {p},\ p<\infty, \ r=1),
			\end{array}\right.
		\end{equation*}
		and when $s=\frac 1 p-1,\ 1\leq p\leq 2,\ r=\infty,\ \text{and}\ V'(t)=\|\nabla v\|_{B^{\frac 1 p}_{p,1}}$.\\
		If $f=v,$ for all $s>0, V^{'}(t)=\|\nabla v(t)\|_{L^{\infty}}.$
	\end{lemm}
	\begin{lemm}\label{cont1}\cite{Chemin2011,Li2016nwpC}
		Let $y_0\in B^{\frac{1}{p}}_{p,1}$ with $1\leq p<\infty,$  and $f\in L^1([0,T];B^{\frac{1}{p}}_{p,1}).$ Define $\bar{\mathbb{N}}=\mathbb{N} \cup {\infty},$ for $n\in \bar{\mathbb{N}},$ denote by $y_n \in C([0,T];B^{\frac{1}{p}}_{p,1})$ the solution of
		\begin{equation}
			\left\{\begin{aligned}
				&\partial_ty_n+(A_n(u)+\Gamma)\partial_xy_n=f,\quad x\in \mathbb{R},\\
				&y_n(t,x)|_{t=0}=y_0(x), \\
			\end{aligned} \right. \label{e1}
		\end{equation}
	where $\Gamma$ is a real number.
		Assume for some $\alpha(t)\in L^1(0,T),\  \sup\limits_{n\in \bar{\mathbb{N}}} \|A_n(u)\|_{B^{1+\frac{1}{p}}_{p,1}}+\Gamma\leq \alpha (t).$ If $A_n(u)$ converges in $A_{\infty}(u)$ in $L^1([0,T];B^{\frac{1}{p}}_{p,1}),$ then the sequence $(y_n)_{n\in \mathbb{N}}$ converges in $C([0,T];B^{\frac{1}{p}}_{p,1}).$
	\end{lemm}
	\section{Local well-posedness}
	\par\
	
	In this section, we consider the local well-posedness of  \eqref{001} in Besov spaces, and   divided it into two cases
	: ${\rm (i)}~s>1+\frac 1 p,~p,r~\in [1,\infty];~{\rm (ii)}~s=1+\frac 1 p,~p\in [1,\infty),~r=1.$ Our main results can be stated as follows.
	\begin{theo}\cite{Chemin2011}\label{exis1}
		Let $u_0\in B^{s}_{p,r}$ with $1\leq p,r \leq \infty,~s>1+\frac 1 p.$ Then, there exists a time $T>0$ such that the equation \eqref{001} has a unique solution  in 
		\begin{equation*}
			E^s_{p,r}(T)\triangleq \left\{\begin{array}{ll}
				C([0,T];B^s_{p,r})\cap C^1([0,T];B^{s-1}_{p,r}), & \text{if}\ r<\infty,  \\
				C_w([0,T];B^s_{p,\infty})\cap C^{0,1}([0,T];B^{s-1}_{p,\infty}), & \text{if}\ r=\infty.
			\end{array}\right.
		\end{equation*}
	Moreover, the solution  depends continuously on the initial data.
	\end{theo}
	
	\begin{theo}\label{local well-pose}
		Let $u_0\in B^{1+\frac 1 p}_{p,1}$ with $p\in [1,\infty).$  Then, there exists a time $T>0$ suct that \eqref{001} has a unique solution $u\in E^p_T\triangleq  C\Big([0,T];B^{1+\frac 1 p}_{p,1}\Big)\cap C^1\Big([0,T];B^{\frac 1 p}_{p,1}\Big).$ Moreover, the solution  depends continuously on the initial data.
	\end{theo}
	\begin{proof}
		\textbf{First Step : Existence of a local solution.}
		\par
		We now construct approximate solutions $(u^n)_{n\in\mathbb{N}}$  which are smooth solutions of the following linear transport equation. Letting $u^0=0,$ we obtain
		\begin{equation}\label{14}
			\left\{\begin{aligned}
				&u^{n+1}_t+(u_n+\Gamma)u_x^{n+1}\triangleq f^n,\\
				&u^{n+1}(t,x)|_{t=0}=S^{n+1} u_0,
			\end{aligned}\right.
		\end{equation}
		with $f^n=\underbrace{\Lambda ^{-2}\partial_x\Big(h(u^n)-(u^n)^2-\frac{1}{2}(u^n_x)^2\Big)}_{Q^n}-\lambda u^n$ and $h(u^n)=(\alpha+\Gamma)u^n+\frac{\beta}{3}(u^n)^3+\frac{\gamma}{4}(u^n)^4.$ 
		Applying $\Delta_j$ to \eqref{14}, we have
		\begin{align}\label{s1}
			\partial_t u_j^{n+1}+u^n\cdot \partial_x u_{j}^{n+1}+\Gamma \partial_x u_{j}^{n+1}=F_j^n+ [u^n\cdot \nabla,\Delta_j]u^{n+1},
		\end{align}
	where $\Delta_j u^{n+1}=u_j^{n+1},~\Delta_j f^n=f_j.$
		Multiplying both sides of the equation \eqref{s1} by ${\rm sgn}( u_j^{n+1})|u_j^{n+1}|^{p-1}$  and integrating over $\mathbb{R}$. Using the fact that $\int_{\mathbb{R}}\partial_x u_{j}^{n+1}\cdot{\rm sgn}( u_j^{n+1})|u_j^{n+1}|^{p-1}dx=0,$ it follows that
		\begin{align*}
			\partial_t\| u_j^{n+1}\|^p_{L^p}-\int_{\mathbb{R}}{ div}u^n | u_j^{n+1}|^pdx&=p\int_{\mathbb{R}}f_j\cdot {\rm sgn}( u_j^{n+1})|u_j^{n+1}|^{p-1}dx\notag\\&~~+p\int_{\mathbb{R}}[u^n\cdot \nabla,\Delta_j]u^{n+1}\cdot{\rm sgn}( u_j^{n+1})|u_j^{n+1}|^{p-1}dx.
		\end{align*}
	 According to  Lemma \ref{priori estimate} , we infer that  
		\begin{align}\label{18}
			1+\|u^{n+1}\|_{B^{1+\frac{1}{p}}_{p,1}}
			&\leq C\exp\Big({C\int_0^t\|u^n(t')\|_{B^{1+\frac{1}{p}}_{p,1}}dt'}\Big)\Big[1+\|u_0\|_{B^{1+\frac{1}{p}}_{p,1}}\notag\\&~~+\int_0^t \exp\Big(-C\int_0^t\|u^n(t'')\|_{B^{1+\frac{1}{p}}_{p,1}}dt''\Big)(1+\|u^n\|_{B^{1+\frac{1}{p}}_{p,1}})^4dt'\Big].
		\end{align}Suppose that $u^n\in L^{\infty}([0,T]B^{1+\frac{1}{p}}_{p,1}).$  Then, we get
		\begin{align}\label{32}
			\|f^n\|_{B^{1+\frac{1}{p}}_{p,1}}\leq C(1+\|u^n\|_{B^{1+\frac{1}{p}}_{p,1}})^4,
		\end{align}
		which leads to  $f^n\in L^{\infty}([0,T];{B^{1+\frac{1}{p}}_{p,1}}).$ Using   Lemmas \ref{existence}-\ref{priori estimate} , we conclude that    there exists a global solution  $u^{n+1}\in C([0,T];{B^{1+\frac{1}{p}}_{p,1}})\cap
		C^1([0,T];{B^{\frac{1}{p}}_{p,1}})$ for all $T>0.$ 
		Fixed a  $T>0$ such that $\frac 2 3 CT(1+\|u_0\|_{B^{1+\frac{1}{p}}_{p,1}})^{\frac 1 2}<1$ and
		\begin{align}\label{19}
			1+\|u^n\|_{B^{1+\frac{1}{p}}_{p,1}}\leq \frac{(1+\|u_0\|_{B^{1+\frac{1}{p}}_{p,1}})^{\frac 1 2}}{1-\frac 2 3 Ct(1+\|u_0\|_{B^{1+\frac{1}{p}}_{p,1}})^{\frac 1 2}}.
		\end{align}
		Plugging \eqref{19} into \eqref{18}, we have
		\begin{align*}
			1+\|u^{n+1}\|_{B^s_{p,r}}\leq \frac{(1+\|u_0\|_{B^s_{p,r}})^{\frac 1 2}}{1-\frac 2 3 Ct(1+\|u_0\|_{B^s_{p,r}})^{\frac 1 2}},~~ \forall ~ t\in [0,T].
		\end{align*}
		Hence, by introduction, we see that $(u^n)_{n\in \mathbb{N}}$ is  bounded in $L^{\infty}([0,T];{B^{1+\frac{1}{p}}_{p,1}})$.
		
		It follows that the compactness method for the approximating sequence $(u^n)_{n\in\mathbb{N}}$ that we get a solution $u$ of \eqref{001}.  Owing to the uniformly boundedness of $u^n$ in $L^{\infty}([0,T];B^{1+\frac{1}{p}}_{p,1}),$ one can get  $\partial_tu^n$ is uniformly bounded in $L^{\infty}([0,T];B^{\frac{1}{p}}_{p,1}).$ Hence, we deduce that 
		$u^n$ is  uniformly bounded in  $C([0,T];B^{1+\frac{1}{p}}_{p,1})\cap C^{\frac 1 2}([0,T];B^{\frac{1}{p}}_{p,1}).$
		On the other hand,
		suppose that $(\phi^j)_{j\in\mathbb{N}}$ be a sequence of smooth functions with value in $[0,1]$ supported in $B(0,j+1)$ and value equal  to $1$ on $B(0,j).$ According to Theorem 2.94 in  \cite{Chemin2011}, we get the map $u^n\mapsto \phi^j u^n$ is compact from $B^{1+\frac{1}{p}}_{p,1}$ to $B^{\frac{1}{p}}_{p,1}$.  Ascoli's theorem entails that there has some function $u^j$ such that, up to extraction,  $(\phi^j u^n)_{j\in\mathbb{N}}$ converges to  $u^j$. And using the Cantor's diagonal process, we deduce that there exists a  subsequence  of $(u^j)_{j\in\mathbb{N}}$ ( still denoted by $(u^j)_{j\in\mathbb{N}}$ ) such that for all $j\in \mathbb{N},$ $\phi^j u^n$ concerges to $u^j$ in $C([0,T]; B^{\frac 1 p}_{p,1}).$  Thanks to $\phi^j\phi^{j+1}=\phi^j,$   we get $u^j=\phi^ju^{j+1}.$ Therefore,  there exists some function $u$  such that for all $\phi\in \mathcal{D}(\mathbb{R}),$ the sequence  $(\phi u^n)_{n\geq 1}$   tends to $                                                                                                                                                                                                                                                                                                                                                                                                                                                                                                                                                                                                                                                                                                                                                                                                                                                                                                                                                                                                                                                                                                                                                                                                                                                                                                                                                                                                                                                                                                                                                                                                                                                                                                                                                                                                                                                                                                                                                                                                                                                                                                                                                                                                                                                                                                                                                                                                                                                                                                                                                                                                                                                                                                                                                                                                                                                                                                                                                                                                                                                                                                                                                                                                                                                                                                                                                                                                                                                                                                                                                                                                                                                                                                                                                                                                                                                                                                                                                                                                                                                                                                                                                                                                                                                                                                                                                                                                                                                                                                                                                                                                                                                                                                                                                                                                                                                                                                                                                                                                                                                                                                                                                                                                                                                                                                                                                                                                                                                                                                                                                                                                                                                                                                                                                                                                                                                                                                                                                                                                                                                                                                                                                                                                                                                                                                                                                                                                                                                                                                                                                                                                                                                                                                                                                                                                                                                                                                                                                                                                                                                                                                                                                                                                                                                                                                                                                                                                                                                                                                                                                                                                                                                                                                                                                                                                                                                                                                                                                                                                                                                                                                                                                                                                                                                                                                                                                                                                                                                                                                                                                                                                                                                                                                                                                                                                                                                                                                                                                                                                                                                                                                                                                                                                                                                                                                                                                                                                                                                                                                                                                                                                                                                                                                                                                                                                                                                                                                                                                                                                                                                                                                                                                                                                                                                                                                                                                                                                                                                                                                                                                                                                                                                                                                                                                                                                                                                                                                                                                                                                                                                                                                                                                                                                                                                                                                                                                                                                                                                                                                                                                                                                                                                                                                                                                                                                                                                                                                                                                                                                                                                                                                                                                                                                                                                                                                                                                                                                                                                                                                                                                                                                                                                                                                                                                                                                                                                                                                                                                                                                                                                                                                                                                                                                                                                                                                                                                                                                                                                                                                                                                                                                                                                                                                                                                                                                                                                                                                                                                                                                                                                                                                                                                                                                                                                                                                                                                                                                                                                                                                                                                                                                                                                                                                                                                                                                                                                                                                                                                                                                                                                                                                                                                                                                                                                                                                                                                                                                                                                                                                                                                                                                                                                                                                                                                                                                                                                                                                                                                                                                                                                                                                                                                                                                                                                                                                                                                                                                                                                                                                                                                                                                                                                                                                                                                                                                                                                                                                                                                                                                                                                                                                                                                                                                                                                                                                                                                                                                                                                                                                                                                                                                                                                                                                                                                                                                                                                                                                                                                                                                                                                                                                                                                                                                                                                                                                                                                                                                                                                                                                                                                                                                                                                                                                                                                                                                                                                                                                                                                                                                                                                                                                                                                                                                                                                                                                                                                                                                                                                                                                                                                                                                                                                                                                                                                                                                                                                                                                                                                                                                                                                                                                                                                                                                                                                                                                                                                                                                                                                                                                                                                                                                                                                                                                                                                                                                                                                                                                                                                                                                                                                                                                                                                                                                                                                                                                                                                                                                                                                                                                                                                                                                                                                                                                                                                                                                                                                                                                                                                                                                                                                                                                                                                                                                                                                                                                                                                                                                                                                                                                                                                                                                                                                                                                                                                                                                                                                                                                                                                                                                                                                                                                                                                                                                                                                                                                                                                                                                                                                                                                                                                                                                                                                                                                                                                                                                                                                                                                                                                                                                                                                                                                                                                                                                                                                                                                                                                                                                \phi u$ in $C([0,T];B^{\frac{1}{p}}_{p,1}).$
		Since the uniform boundness of $u^n$ and Fatou property guarantee that $u\in L^{\infty}([0,T];B^{1+\frac{1}{p}}_{p,1}).$ Making use of interpolation, we get $\phi u^n$ tends to $\phi u$ in $C([0,T];B^{1+\frac{1}{p}-\varepsilon}_{p,1})$ for any $\varepsilon>0.$  
		
		We now check that $u$ is the solution of the equation \eqref{001}. For fixed $\psi \in B^{-\frac{1}{p}}_{p',\infty},$ we have
		\begin{align}\label{diff}
			\left\langle \partial_t(\phi u^n)-\partial_t(\phi u),\psi\right\rangle-	\left\langle (\phi u^n+\Gamma)\partial_x(\phi u^n)-((\phi u)+\Gamma)\partial_x(\phi u),\psi\right\rangle-\left\langle  F(\phi u^n)-F(\phi u),\psi\right\rangle\xrightarrow{n\rightarrow \infty} 0.
		\end{align}
	The main difficulty is to  prove that  $	\left\langle\Lambda^{-2}\partial_x\Big((\phi u^n)_x^2-(\phi u)_x^2\Big),\psi\right\rangle\xrightarrow{n\rightarrow \infty} 0$  in \eqref{diff}. For simplicity, we only handle the  $	\left\langle\Lambda^{-2}\partial_x\Big((\phi u^n)_x^2-(\phi u)_x^2\Big),\psi\right\rangle\stackrel{n\rightarrow \infty}{\longrightarrow} 0$.  The others are similar. Hence, 
		\begin{align}\label{oooo}
			|\left\langle\Lambda^{-2}\partial_x\Big((\phi u^n)_x^2-(\phi u)_x^2\Big),\psi\right\rangle&|\leq |\left\langle\Lambda^{-2}\partial_x\Big((\phi u^n)_x-(\phi u)_x\Big)\Big((\phi u^n)_x+(\phi u)_x\Big),\psi\right\rangle|\notag\\
			&\leq \|\Lambda^{-2}\partial_x\{\Big((\phi u^n)_x-(\phi u)_x\Big)\Big((\phi u^n)_x+(\phi u)_x\Big)\}\|_{B^{\frac 1 p}_{p,1}}\|\psi\|_{B^{-\frac 1 p}_{p',\infty}}\notag\\
			&\leq\|\Big((\phi u^n)_x-(\phi u)_x\Big)\Big((\phi u^n)_x+(\phi u)_x\Big)\|_{B^{\frac 1 p -1}_{p,1}}\|\psi\|_{B^{-\frac 1 p}_{p',\infty}}\notag\\
			&\leq\|\Big((\phi u^n)_x-(\phi u)_x\Big)\Big((\phi u^n)_x+(\phi u)_x\Big)\|_{B^{\frac 1 p -\epsilon}_{p,1}}\|\psi\|_{B^{-\frac 1 p}_{p',\infty}}\notag\\
			&
			\leq C\|(\phi u^n)_x+(\phi u)_x\|_{L^{\infty}}\|(\phi u^n)_x-(\phi u)_x\|_{B^{\frac{1}{p}-\epsilon}_{p,1}}\|\psi\|_{B^{{-\frac{1}{p}}}_{p',\infty}}\\ \notag&\leq C\|\phi u^n+\phi u\|_{B^{1+\frac{1}{p}-\epsilon}_{p,1}}\|\phi  u^n-\phi u\|_{B^{1+\frac{1}{p}-\epsilon}_{p,1}}\|\psi\|_{B^{{-\frac{1}{p}}}_{p',\infty}}.
		\end{align}
		By using the fact that  $\phi u^n\xrightarrow{n\rightarrow \infty} \phi u$ in $C([0,T];B^{1+\frac{1}{p}-\epsilon}_{p,1}),$  and that  $\phi u^n$ is bounded in $L^{\infty}([0,T];B^{1+\frac{1}{p}}_{p,1}),$   we get
		\eqref{oooo} tends to $0$ uniformly on $[0,T]$ as $n\rightarrow\infty$  and   $u\in E^p_T.$
		
		\textbf{Second Step: Uniqueness.}
		\par
		In order to prove the uniqueness of solution, we reformulate the equation using a transformation
		that corresponds to the transformation between Eulerian and Lagrangian coordinates. Let $u=u(t,x)$ denote the solution, and $t\mapsto y(t,\xi)$ be the characteristic as follows
		\begin{equation}
			\left\{\begin{aligned}
				&y_t(t,\xi)=u(t,y(t,\xi))+\Gamma,\quad x\in \mathbb{R},\\
				&y(t,\xi)|_{t=0}=\bar{y}(\xi). \\
			\end{aligned} \right. \label{02}
		\end{equation}
		Our  new variables  are $y(t,\xi)$, $U(t,\xi)=u(t,y(t,\xi)), $ then the equation \eqref{001} can be rewritten as
		\begin{align}\label{05}
			U_t(t,\xi)=u_t(t,y(t,\xi))+y_t(t,\xi)u_x(t,y(t,\xi))=\tilde{Q}(t,\xi)-\lambda U(t,\xi),
		\end{align}
		where
		\begin{align}\label{o1}
			\tilde{Q}(t,\xi)\triangleq P_x\circ y(t,\xi)=\frac 1 2 \int_{\mathbb{R}}{\rm sgn}(y(t,\xi)-x)e^{-|y(t,\xi)-x|}\Big(-h(u)+u^2+\frac{1}{2} u_x^2\Big)(t,x)dx.
		\end{align}
		After the change of variables $x=y(t,\eta),$ we have 
		\begin{align}\label{o2}
			\tilde{Q}(t,\xi)=\frac 1 2\int_{\mathbb{R}}{\rm sgn}(y(t,\xi)-y(t,\eta))e^{-|y(t,\xi)-y(t,\eta)|}\Big(-h(U)y_{\xi}+U^2y_{\xi}+\frac{U^2_{\xi}}{y_{\xi}}\Big)(t,\eta)d\eta.
		\end{align}
		For fixed $t, ~y(t,\cdot)$ is an increasing function, which implies that ${\rm sgn}(y(t,\xi)-y(t,\eta))={\rm sgn}(\xi-\eta).$ That is,
		\begin{align}\label{o7}
			\tilde{Q}(t,\xi)=&\frac 1 2\int_{\mathbb{R}}{\rm sgn}(\xi-\eta)e^{-|y(t,\xi)-y(t,\eta)|}\Big(-h(U)y_{\xi}+U^2y_{\xi}+\frac{U^2_{\xi}}{y_{\xi}}\Big)(t,\eta)d\eta\notag\\
			=&\frac 1 2\Big(\int_{-\infty}^{\xi}-\int_{\xi}^{+\infty}\Big)e^{-|y(t,\xi)-y(t,\eta)|}\Big(-h(U)y_{\xi}+U^2y_{\xi}+\frac{U^2_{\xi}}{y_{\xi}}\Big)(t,\eta)d\eta.
		\end{align}
		We now give another new variable $\zeta$  defined as $\zeta(t,\xi)=y(t,\xi)-\xi-\Gamma t,$ and $\zeta$ satisfies 
		\begin{align}
			\zeta_t(t,\xi)&=U(t,\xi). \label{o5}
		\end{align}
		Therefore,  the derivatives of $\tilde{Q}$ is given by
		\begin{align}\label{o3}
			\tilde{Q}_{\xi}=(-h(U)+U^2-P)(\zeta_{\xi}+1)+\frac{U^2_{\xi}}{2y_{\xi}},
		\end{align}
		with $$P(t,\xi)=-\frac 1 2\int_{\mathbb{R}}e^{-|(y(t,\xi)-y(t,\eta))|}\Big(-h(U)y_{\xi}+U^2y_{\xi}+\frac{U^2_{\xi}}{y_{\xi}}\Big)(t,\eta)d\eta.$$
		From the above equation, we can deduce that
		\begin{align}\label{o4}
			U_{t\xi}(t,\xi)=&\tilde{Q}_{\xi}(t,\xi)-\lambda U_{\xi}.
		\end{align}
	Combining  \eqref{02} and \eqref{b1}, we  infer that   $\zeta(t,\xi)$ satisfies the following integral form
		\begin{align}
			&\zeta(t,\xi)~=\int_0^tU(\tau,\xi)d\tau,\label{b0}\\
			&\zeta_{\xi}(t,\xi)=\int_0^tU_{\xi}(\tau,\xi)d\tau.\label{b1}
		\end{align}
		Note that $\zeta_{\xi}+1=y_{\xi}.$
		Using the uniform boundedness of   $u$   in $C([0,T]; B^{1+\frac{1}{p}}_{p,1})$ and the embedding   with $B^{1+\frac{1}{p}}_{p,1}\hookrightarrow W^{1,p}\cap W^{1,\infty},$ we get $u\in  C([0,T];W^{1,p}\cap W^{1,\infty} ).$ Thereby, we obtain that $\zeta_{\xi}$ is uniformly bounded in $L^{\infty}([0,T]; L^{\infty}),$    which implies   $y_{\xi}\in L^{\infty}([0,T];L^{\infty}).$ 
		More importantly, we have  $\frac{1}{2}\leq \zeta_{\xi}, y_{\xi}\leq C_{u_0}$  for $T>0$ small enough. Without loss of generality,  suppose that $t$ is sufficiently small, otherwise we use the continuity method. Now, we  prove that  $U(t,\xi)$ is bounded in $L^{\infty}([0,T]; W^{1,p})$ as follows.
		\begin{align*}
			&\|U\|^p_{L^p}=\int_{\mathbb{R}} |U(t,\xi)|^pd\xi=\int_{\mathbb{R}} |u(t,y(t,\xi))|^p\frac{1}{y_{\xi}}dy\leq \|u\|^p_{L^p}\|\frac{1}{y_{\xi}}\|_{L^{\infty}}\leq 2\|u\|^p_{L^p}\leq C;\\
			&\|U_{\xi}\|^p_{L^p}=\int_{\mathbb{R}} |U_{\xi}|^p|y_{\xi}|^{p-1}d\xi=\int_{\mathbb{R}} |u_x(t,y(t,\xi))|^p|y_{\xi}|^{p-1}dy\leq \|u_x\|^p_{L^p}\|{y_{\xi}}\|^{p-1}_{L^{\infty}}\leq C^{p-1}_{u_0}\|u_x\|^p_{L^p}\leq C.
		\end{align*}
		Applying the above inequalities, we get $U(t,\xi)\in L^{\infty}( [0,T];W^{1,p}\cap W^{1,\infty}),~\zeta \in L^{\infty}([0,T]; W^{1,p}\cap W^{1,\infty})$ and $\frac{1}{2}\leq \zeta_{\xi},~ y_{\xi}\leq C_{u_0}$ for any $t\in [0,T].$
		
		Let $u_1, u_2$ be two solutions of \eqref{001}, and $U_i(t,\xi)=u(t,y_i(t,\xi))\ (i=1,2)$ satisfy
		\begin{align}
			U_{it}(t,\xi)=u_{it}(t,y_i)+y_{it}(t,\xi)u_{ix}(t,y_i(t,\xi))=\tilde{Q}_i(t,\xi)-\lambda U_{i}(t,\xi).\label{21}
		\end{align}
		Likewise, we have $U_i(t,y_i(t,\xi))\in L^{\infty}([0,T]; W^{1,p}\cap W^{1,\infty}),~ \zeta_i(t,\xi)\in L^{\infty}([0,T]; W^{1,p}\cap W^{1,\infty})$ and $\frac{1}{2}\leq \zeta_{i\xi},~y_{i\xi}\leq C_{u_0}$ for sufficiently small $T>0.$
		
		We now demonstrate the following estimate.  $$\|\tilde{Q}_{1}(t,\xi)-\tilde{Q}_{2}(t,\xi)\|_{W^{1,p}\cap W^{1,\infty}} \leq C(\|U_1-U_2\|_{W^{1,p}\cap W^{1,\infty}}+\|y_1-y_2\|_{W^{1,p}\cap W^{1,\infty}}).$$
		According to \eqref{o7},  we can deduce that 
		\begin{align}\label{q1}
			~~~&\tilde{Q}_{1}(t,\xi)-\tilde{Q}_{2}(t,\xi)\notag\\
			&=
			\frac{1}{4}\Big(\int_{-\infty}^{\xi}-\int_{\xi}^{+\infty}\Big)e^{-|y_{2}(\xi)-y_{2}(\eta)|}\Big(\frac{U^2_{1\eta}}{y_{1\eta}}-\frac{U^2_{2\eta}}{y_{2\eta}}\Big)(\eta)d\eta\notag\\
			~~~&~~~+\frac{1}{4}\Big(\int_{-\infty}^{\xi}-\int_{\xi}^{+\infty}\Big)\Big(e^{-|y_{1}(\xi)-y_{1}(\eta)|}-e^{-|y_{2}(\xi)-y_{2}(\eta)|}\Big)\frac{U^2_{1\eta}}{y_{1\eta}} (\eta)d\eta\notag\\ 
			~~~&~~~+\frac{1}{2}\Big(\int_{-\infty}^{\xi}-\int_{\xi}^{+\infty}\Big)\Big(e^{-|y_{1}(\xi)-y_{1}(\eta)|}-e^{-|y_{2}(\xi)-y_{2}(\eta)|}\Big)\Big(-h(U_1)y_{1\eta}+U_1^2y_{1\eta}\Big) (\eta)d\eta\notag\\
			~~~&~~~+\frac{1}{2}\Big(\int_{-\infty}^{\xi}-\int_{\xi}^{+\infty}\Big)e^{-|y_{2}(\xi)-y_{2}(\eta)|}\Big(h(U_2)y_{2\eta}-h(U_1)y_{1\eta}+U_1^2y_{1\eta}-U_2^2y_{2\eta}\Big)(\eta)d\eta\notag\\
			&
			\triangleq \sum_{i=1}^4\mathcal{I}_i.
		\end{align}
		We first give an estimate of  $\mathcal{I}_1$  as follows.
		\begin{align}\label{q2}
			\mathcal{I}_1=&\frac{1}{4}\int_{-\infty}^{\xi}e^{-(y_{2}(\xi)-y_{2}(\eta))}\Big(\frac{U^2_{1\eta}}{y_{1\eta}}-\frac{U^2_{2\eta}}{y_{2\eta}}\Big) (\eta)d\eta-\frac{1}{4}\int_{\xi}^{+\infty}e^{y_{2}(\xi)-y_{2}(\eta)}\Big(\frac{U^2_{1\eta}}{y_{1\eta}}-\frac{U^2_{2\eta}}{y_{2\eta}}\Big) (\eta)d\eta
			\notag\\
			=&\frac{1}{4}\int_{-\infty}^{\xi}e^{-(\xi-\eta)}e^{-\int_0^t(U_2(\xi)-U_2(\eta))(\tau)d\tau}\Big(\frac{U^2_{1\eta}-U^2_{2\eta}}{y_{1\eta}}+\frac{U^2_{2\eta}(y_{2\eta}-y_{1\eta})}{y_{1\eta}y_{2\eta}}\Big) (\eta)d\eta\notag\\&-
			\frac{1}{4}\int_{\xi}^{+\infty}e^{\xi-\eta}e^{\int_0^t(U_2(\xi)-U_2(\eta))(\tau)d\tau}\Big(\frac{U^2_{1\eta}-U^2_{2\eta}}{y_{1\eta}}+\frac{U^2_{2\eta}(y_{2\eta}-y_{1\eta})}{y_{1\eta}y_{2\eta}}\Big) (\eta)d\eta.
		\end{align}
	Notice that $|e^{-|x|}-e^{-|y|}|\leq \max\{e^{-|x|},e^{-|y|}\}|x-y|,$ and $\frac{1}{2}\leq \zeta_{\xi}, ~y_{\xi}\leq C_{u_0}$ for $T>0$ small enough. According to $U_{i}(t,\xi)$ and $y_{i}(t,\xi)$ are bounded in $L^{\infty}_{T}( W^{1,\infty}),$ we get
		\begin{align*}
			\mathcal{I}_1 &\leq C\Big(\|\frac{U_{1\eta}+U_{2\eta}}{y_{1\eta}}\|_{L^{\infty}}+\|\frac{U^2_{2\eta}}{y_{1\eta}y_{2\eta}}\|_{L^{\infty}}\Big)\int_{\xi}^{+\infty}e^{\xi-\eta}\Big(|U_{1\eta}-U_{2\eta}|+|y_{1\eta}-y_{2\eta}|\Big)(\eta)d\eta\notag\\
			&~~+C\Big(\|\frac{U_{1\eta}+U_{2\eta}}{y_{1\eta}}\|_{L^{\infty}}+\|\frac{U^2_{2\eta}}{y_{1\eta}y_{2\eta}}\|_{L^{\infty}}\Big)\int_{-\infty}^{\xi}e^{-(\xi-\eta)}\Big(|U_{1\eta}-U_{2\eta}|+|y_{1\eta}-y_{2\eta}|\Big)(\eta)d\eta\notag\\
			&\leq C\Big[1_{\geq 0}(x)e^{-|x|}\ast \Big(|U_{1\eta}-U_{2\eta}|+|y_{1\eta}-y_{2\eta}|\Big)+1_{\leq 0}(x)e^{-|x|}\ast \Big(|U_{1\eta}-U_{2\eta}|+|y_{1\eta}-y_{2\eta}|\Big)\Big].
		\end{align*}
		Since the other terms  can be proved similarly, then we  conclude that
		\begin{align*}
			\tilde{Q}_{1}(t,\xi)-\tilde{Q}_{2}(t,\xi)& \leq C\Big[1_{\geq 0}(x)e^{-|x|}\ast \Big(|U_1-U_2|+|U_{1\eta}-U_{2\eta}|+|y_{1\eta}-y_{2\eta}|\Big)\Big]\notag\\
			&~~~+C\Big[1_{\leq 0}(x)e^{-|x|}\ast \Big(|U_1-U_2|+|U_{1\eta}-U_{2\eta}|+|y_{1\eta}-y_{2\eta}|\Big)\Big]\notag\\
			&~~~+C\|U_1-U_2\|_{L^{\infty}}\Big[1_{\geq 0}(x)e^{-|x|}\ast\Big(h(U_1)y_{1\eta}+\frac{U^2_{1\eta}}{y_{1\eta}}+U^2_{1\eta}y_{1\eta}\Big)\notag\\
			&~~~+1_{\leq 0}(x)e^{-|x|}\ast\Big(h(U_1)y_{1\eta}
			+\frac{U^2_{1\eta}}{y_{1\eta}}+U^2_{1\eta}y_{1\eta}\Big)\Big].
		\end{align*}
	That is
		\begin{align}\label{pp2}	
			\|\tilde{F}_{1}-\tilde{F}_{2}\|_{L^{\infty}\cap L^p}\leq C\Big(\|U_1-U_2\|_{L^{\infty}\cap L^p}+\|U_{1\eta}-U_{2\eta}\|_{L^{\infty}\cap L^p}+\|y_{1\eta}-y_{2\eta}\|_{L^{\infty}\cap L^p}\Big).
		\end{align}
		It is now clear that the same estimate holds for $\|\tilde{F}_{1\xi}-\tilde{F}_{2\xi}\|_{L^{\infty}\cap L^p}$ as follows
		\begin{align}\label{25}
			\|\tilde{F}_{1\xi}-\tilde{F}_{2\xi}\|_{L^{\infty}\cap L^p}\leq C\Big(\|U_1-U_2\|_{L^{\infty}\cap L^p}+\|U_{1\eta}-U_{2\eta}\|_{L^{\infty}\cap L^p}+\|y_{1\eta}-y_{2\eta}\|_{L^{\infty}\cap L^p}\Big).
		\end{align}
		Making use of \eqref{pp2} and \eqref{25} yields that 
		\begin{align}\label{26}
			\|\tilde{F}_{1}(t,\xi)-\tilde{F}_{2}(t,\xi)\|_{W^{1,p}\cap W^{1,\infty}}\leq C(\|U_1-U_2\|_{W^{1,p}\cap W^{1,\infty}}+\|y_1-y_2\|_{W^{1,p}\cap W^{1,\infty}}).
		\end{align}
		From the above analysis, we get
		\begin{align}\label{27}
			&\|U_1-U_2\|_{W^{1,p}\cap W^{1,\infty}}+\|y_1-y_2\|_{W^{1,p}\cap W^{1,\infty}}\notag\\ &\leq C \Big(\|U_1(0)-U_2(0)\|_{W^{1,p}\cap W^{1,\infty}}+\|y_1(0)-y_2(0)\|_{W^{1,p}\cap W^{1,\infty}}\Big)\notag\\
			&~~~+C\int_0^T\Big(\|U_1-U_2\|_{W^{1,p}\cap W^{1,\infty}}+\|y_1-y_2\|_{W^{1,p}\cap W^{1,\infty}}(t)\Big)dt.
		\end{align}
		Plugging \eqref{26} into \eqref{27} and using the fact that  $y_1(0)=y_2(0)=\xi,$   we infer that
		\begin{align*}
			\|U_1-U_2&\|_{W^{1,p}\cap W^{1,\infty}}+\|y_1-y_2\|_{W^{1,p}\cap W^{1,\infty}}\notag\\ &\leq C e^{CT}(\|U_1(0)-U_2(0)\|_{W^{1,p}\cap W^{1,\infty}}+\|y_1(0)-y_2(0)\|_{W^{1,p}\cap W^{1,\infty}})\notag\\ &\leq C e^{CT}(\|U_1(0)-U_2(0)\|_{W^{1,p}\cap W^{1,\infty}}+0)\notag\\ &\leq Ce^{CT} \|u_1(0)-u_2(0)\|_{B^{1+\frac{1}{p}}_{p,1}},
		\end{align*}
		from which it follows,
		\begin{align*}
			&\|u_1-u_2\|_{L^p}\leq C \|u_1 \circ y_1-	u_2 \circ y_2\|_{L^p}\notag\\  &~~~~~\leq C\|u_1 \circ y_1-u_2 \circ y_1+u_2 \circ y_1-u_2 \circ y_2\|_{L^p} \notag\\ &~~~~~\leq C\|U_1-U_2\|_{L^p}+C\|u_{2x}\|_{L^{\infty}}\|y_1-y_2\|_{L^p}\notag\\ &~~~~~\leq C \|u_1(0)-u_2(0)\|_{B^{1+\frac{1}{p}}_{p,1}}.
		\end{align*}
		The embedding $L^p\hookrightarrow B^{0}_{p,\infty}$ guarantees that
		\begin{align*}
			\|u_1-u_2\|_{B^{0}_{p,\infty}}\leq C\|u_1-u_2\|_{L^p}\leq C\|u_1(0)-u_2(0)\|_{B^{1+\frac{1}{p}}_{p,1}}.
		\end{align*}
		This completes the proof of the uniqueness of the solution of  \eqref{001} if $u_1(0)=u_2(0).$ \\
		
		\textbf{Third Step: The continuous dependence.}
		
		Assume that we are given $u^n$ and $u^{\infty}$, two  solutions of  \eqref{001} with the initial data $u^n_0$ and $u^{\infty}_0$ satisfying  $u^n_0$ tends to $u^{\infty}_0$ in $B^{1+\frac1 p }_{p,1}.$ The \textbf{First Step} and  \textbf{Second Step} guarantee that $u^n$ and  $u^{\infty}$ are uniformly bounded in $L^{\infty}([0,T]; B^{1+\frac1 p }_{p,1}),$ and
		\begin{align*}
			\|(u^n-u^{\infty})(t)\|_{B^{0 }_{p,\infty}}\leq C\|u^n_0-u^{\infty}_0\|_{B^{1+\frac1 p }_{p,1}}.
		\end{align*}
	This implies $u^n$ converges to $u^{\infty}$ in $C([0,T];B^{0 }_{p,\infty}).$ The interpolation ensures that $u^n\rightarrow u^{\infty}$ in $C([0,T];B^{1+\frac{1}{p}-\varepsilon}_{p,1})$ for any $\varepsilon >0$. If $\varepsilon =1$, we have
		\begin{align}\label{a0}
			u^n\rightarrow u^{\infty} \ in \ C([0,T];B^{\frac{1}{P}}_{p,1}).
		\end{align}
		In order to get $u^n$ tends to  $u^{\infty}$ in $C([0,T];{B^{1+\frac1 p }_{p,1}})$, we now prove that  $u_x^n$ tends to  $u_x^{\infty}$ in $C([0,T];{B^{\frac1 p }_{p,1}}).$ Letting $v^n=u_x^n,$  we split $v^n=z^n+w^n$  with $(z^n,w^n)$ satisfying
		\begin{equation*}
			\left\{\begin{array}{l}
				w^n_t+u^nw_x^n+\Gamma w_x^n=f^{\infty},\\ w^n(t,x)|_{t=0}=u^{\infty}_{0x},  \\
			\end{array}\right.
			\text{and} \quad
			\left\{\begin{array}{l}
				z^n_t+u^n z_x^n+\Gamma z_x^n=F^n-f^{\infty},\\
				z^n(t,x)|_{t=0}=u^n_{0x}-u^{\infty}_{0x},\\
			\end{array}\right.
		\end{equation*}
		with
		\begin{align*}
			f^n=-\lambda v^n-h(u^n)+(u^n)^2-\frac 1 2 (v^n)^2+\Lambda^{-2}\Big(h(u^n)-(u^n)^2-\frac 1 2 (v^n)^2\Big),~\forall~n\in\bar{\mathbb{N}}.
		\end{align*}
		Combining Lemma \ref{R} with the boundeness of $u^n$ and $u^{\infty}$ in $C([0,T];B^{1+\frac1 p}_{p,1}),$ we deduce that $\|u^n\|_{B^{1+\frac 1 p}_{p,1}}\leq C,$ and 
		\begin{align}\label{con0}
			\|f^n-f^{\infty}\|_{B^{\frac{1}{p}}_{p,1}}&\leq C\|u^n-u^{\infty}\|_{B^{1+\frac{1}{p}}_{p,1}}\Big(1+\|u^n\|_{B^{1+\frac{1}{p}}_{p,1}}+\|u^{\infty}\|_{B^{1+\frac{1}{p}}_{p,1}}\Big)^3 \notag\\
			&\leq C\Big(\|u^n-u^{\infty}\|_{B^{\frac{1}{p}}_{p,1}}+\|u_x^n-u_x^{\infty}\|_{B^{\frac{1}{p}}_{p,1}}\Big)\notag\\
			&\leq C\Big(\|u^n-u^{\infty}\|_{B^{\frac{1}{p}}_{p,1}}+\|z^n\|_{B^{\frac{1}{p}}_{p,1}}+\|w^n-w^{\infty}\|_{B^{\frac{1}{p}}_{p,1}}\Big).
		\end{align}
		Using Lemma \ref{priori estimate} and \eqref{con0}, for any $n\in\bar{\mathbb{N}}, t\in [0,T],$ we obtain
		\begin{align}\label{z1}
			\|z_n(t)\|_{B^{\frac{1}{p}}_{p,1}}&\leq C\Big(\|u^n_0-u^{\infty}_0\|_{B^{\frac{1}{p}}_{p,1}}+\int_0^t\|f^n-f^{\infty}\|_{B^{1+\frac{1}{p}}_{p,1}}dt'\Big)\\ \notag
			&\leq C\Big(\|u^n_0-u^{\infty}_0\|_{B^{\frac{1}{p}}_{p,1}}+\int_0^t\Big(\|u^n-u^{\infty}\|_{B^{\frac{1}{p}}_{p,1}}+\|z^n\|_{B^{\frac{1}{p}}_{p,1}}+\|w^n-w^{\infty}\|_{B^{\frac{1}{p}}_{p,1}}\Big)(t')dt',
		\end{align}
		where we used the fact  that $\int_{\mathbb{R}}\partial_x\Delta_jz\cdot{\rm sgn}({\Delta_jz})|\Delta_jz|^{p-1}dx=0$ when the linear term $\Gamma z_x$ is estimated in $B^{\frac 1 p}_{p,1}.$
		Hence, we have 
		\begin{align*}
			\|z^n\|_{B^{\frac{1}{p}}_{p,1}}\leq e^{Ct}\Big(\|v^n_0-v^{\infty}_0\|_{B^{\frac{1}{p}}_{p,1}}+\int_0^te^{-Ct'}(\|w^n-w^{\infty}\|_{B^{\frac{1}{p}}_{p,1}}+\|z^n\|{B^{\frac{1}{p}}_{p,1}})(t')dt'\Big).
		\end{align*}
		Notice  that
		\begin{align*}
			~& u^n\rightarrow u^{\infty}~~in ~C([0,T];{B^{\frac{1}{p}}_{p,1}});\notag\\
			~& v^n_0 \rightarrow v^{\infty}_0~~in ~C([0,T];{B^{\frac{1}{p}}_{p,1}});\notag\\
			~& w^n_0 \rightarrow w^{\infty}_0~in ~C([0,T];{B^{\frac{1}{p}}_{p,1}}).
		\end{align*}  
		Hence,  we have $z^n\rightarrow 0$ in $C([0,T];{B^{\frac{1}{p}}_{p,1}})$ as $n\rightarrow \infty.$  Lemmas \ref{existence}-\ref{priori estimate} and  $z^{\infty}=0$ entail that $z^n$ converges to  $z^{\infty}$ in  $C([0,T];{B^{\frac{1}{p}}_{p,1}}).$ Moreover, we thus conclude that $u^n\rightarrow u^{\infty}$ in $C([0,T];{B^{1+\frac{1}{p}}_{p,1}}).$   
		
		Thus, we complete the proof of Theorem \ref{local well-pose} .
	\end{proof}
	
	\section{Global existence and blow-up}
	\par
		$~~$In this section, we consider the global existence and blow-up of solutions for \eqref{001}.
	\subsection{Global existence for small data}
	\par
	We first prove the  global existence for small data.
	\begin{theo}\label{global}
		Let $u_0 \in B^{s}_{p,r}$ with  $(s,p,r)$ being as  in Theorems \ref{exis1}-\ref{local well-pose} . There is a  constant $\epsilon$ such that if
		\begin{align}\label{ss0}
			H_0	\triangleq|\alpha|+|\Gamma|+\|u_0\|_{B^s_{p,r}}+\frac{|\beta|}{3}\|u_0\|^2_{B^s_{p,r}}+\frac{|\gamma|}{4}\|u_0\|^3_{B^s_{p,r}}\leq \lambda \epsilon,
		\end{align}
		for all real numbers $\alpha, \beta, \gamma, \Gamma$ and $\lambda>0,$ then, there exists a unique global solution $u$ of the equation \eqref{001}, and for any $t\in [0,\infty)$ we have
		\begin{align}\label{small}
			H (t)\triangleq|\alpha|+|\Gamma|+\|u(t)\|_{B^s_{p,r}}+\frac{|\beta|}{3}\|u(t)\|^2_{B^s_{p,r}}+\frac{|\gamma|}{4}\|u(t)\|^3_{B^s_{p,r}}\leq H_0.
		\end{align}
	\end{theo}
	\begin{proof}
		According to   \eqref{001} and the
		proofs of Theorems \ref{exis1}- \ref{local well-pose} , for any $t\in[0,T],$ we get
		\begin{align}\label{s2}
			&\|u(t)\|_{{B^s_{p,r}}}+\lambda \|u\|_{{L^1_t(B^s_{p,r})}}\notag\\
			&\leq \|u_0\|_{B^s_{p,r}}+C\int_0^t \|u(t')\|_{B^s_{p,r}}(|\alpha|+|\Gamma|+\|u(t')\|_{B^s_{p,r}}+\frac{|\beta|}{3}\|u(t')\|^2_{B^s_{p,r}}+\frac{|\gamma|}{4}\|u(t')\|^3_{B^s_{p,r}})dt',
		\end{align}
		where we used  Lemmas \ref{existence}-\ref{priori estimate} .  For any $\epsilon>0$ sufficiently small,  suppose that for all $t\in [0,T],$  we have
		\begin{align}\label{s4}
			|\alpha|+|\Gamma|+\|u(t)\|_{L^{\infty}_T(B^s_{p,r})}+\frac{|\beta|}{3}\|u(t)\|^2_{{L^{\infty}_T(B^s_{p,r})}}+\frac{|\gamma|}{4}\|u(t)\|^3_{{L^{\infty}_T(B^s_{p,r})}}\leq 2\lambda \epsilon.
		\end{align}
		This means
		\begin{align}\label{s5}
			\int_0^T \|u\|_{B^s_{p,r}}(	|\alpha|+|\Gamma|+\|u(t)\|_{B^s_{p,r}}+\frac{|\beta|}{3}\|u(t)\|^2_{B^s_{p,r}}+\frac{|\gamma|}{4}\|u(t)\|^3_{B^s_{p,r}})dt\leq 2\lambda \epsilon	\int_0^T \|u(t)\|_{B^s_{p,r}}dt,
		\end{align}
		for some sufficiently small $\epsilon.$
		
		Substituting \eqref{s5} into \eqref{s2} yields
		\begin{align}\label{s6}
			\|u(t)\|_{B^s_{p,r}}\leq \|u_0\|_{B^s_{p,r}},~\forall ~t\in [0,T].
		\end{align}
		Analogously, we infer that
		\begin{align}\label{s7}
			&|\alpha|+|\Gamma|+\|u(t)\|_{L^{\infty}_T(B^s_{p,r})}+\frac{|\beta|}{3}\|u(t)\|^2_{{B^s_{p,r}}}+\frac{|\gamma|}{4}\|u(t)\|^3_{{B^s_{p,r}}}
			\notag\\
			&~~~~~~~~+\lambda\int_0^t(\|u(t')\|_{B^s_{p,r}}+\frac{\beta}{3}\|u(t')\|^2_{B^s_{p,r}}+\frac{\gamma}{4}\|u(t')\|^4_{B^s_{p,r}})dt'\notag\\
			&\leq	H_0+C\int_0^t(|\alpha|+|\Gamma|+\|u(t')\|_{B^s_{p,r}}+\frac{|\beta|}{3}\|u(t')\|^2_{B^s_{p,r}}+\frac{\gamma}{4}\|u(t')\|^4_{B^s_{p,r}})\notag\\
			&~~~~~~~~~\cdot(|\alpha|+|\Gamma|+\|u(t')\|_{L^{\infty}_T(B^s_{p,r})}+\frac{|\beta|}{3}\|u(t')\|^2_{{L^{\infty}_T(B^s_{p,r})}}+\frac{|\gamma|}{4}\|u(t')\|^3_{{L^{\infty}_T(B^s_{p,r})}})dt'.
		\end{align}
		Combining Theorems \ref{exis1}-\ref{local well-pose} with \eqref{s4}-\eqref{s7}, we have the unique global solution on time initival $[0,T],$ and	$\sup_{t\in[0,T]}H(t)\leq H_0.$
Moreover, we have
	$$H(T)\leq H_0.$$
		Applying Theorems \ref{exis1}-\ref{local well-pose} again, we have the unique local solution  and   $ H(t)\leq \lambda \epsilon$  on $t\in [T, 2T].$ From \eqref{s4}-\eqref{s6}, we get
		$\sup_{t\in[T,2T]}H(t)\leq H_0.$ Hence, we deduce that
		$$\sup_{[0,2T]}H(t)\leq H_0.$$
		 
		 Repeating the bootstrap argument, we prove the global existence  of  \eqref{001}.
	\end{proof}
	\subsection{Blow-up  criteria}
	\par
	In this subsection, we present two blow-up criteria for the equation 
	\eqref{001}.
	\begin{lemm}\cite{Igor02decay}\label{conser}
		Let $u_0\in H^s$ with $s>\frac 3 2,$ and let $T^*$ be the maximal existence time of  the corresponding solution $u$ to \eqref{001}. For any $t\in[0,T^*),$ we have
		$$\|u(t) \|_{H^1}=e^{-\lambda t}\|u_0\|_{H^1}.$$
	\end{lemm}
	\begin{lemm}\label{blowup cri}
		Let $u\in B^s_{p,r}$ with $(s,p,r)$ being as in Theorems \ref{exis1}-\ref{local well-pose} , and let $T^*$ be the maximal existence time of the corresponding solution $u$  to \eqref{001}. Then the  solution  of  \eqref{001} blows up in finite time $T^*<\infty$ if and only if
		$$\int_0^{T^*}\|u_x(t')\|_{L^{\infty}}dt'=\infty.$$
	\end{lemm}
	\begin{proof}
		Combining  \eqref{001} with Lemma \ref{R} , it follows that
		\begin{align}\label{f0}
			&\|u(t)\|_{B^s_{p,r}}+\lambda 	\|u(t)\|_{L^1_t([0,T];B^s_{p,r})}\notag\\ &\leq \|u_0\|_{B^s_{p,r}}+C\int_{0}^t\Big(\|u_x\|_{L^{\infty}}\|u\|_{B^s_{p,r}}+\|u\|_{B^{s}_{p,r}}+\|u_x\|_{B^{s-1}_{p,r}}\|u_x\|_{B^{0}_{\infty,\infty}}+\|Q\|_{B^{s}_{p,r}}\Big)(t')dt'\notag\\& \leq \|u_0\|_{B^s_{p,r}}+C\int_{0}^t\Big(\underbrace{\|u_x\|_{L^{\infty}}\|u\|_{B^s_{p,r}}+\|Q\|_{B^{s}_{p,r}}}_{G}\Big)(t')dt'.
		\end{align}
	Moreover, we have
		\begin{align}\label{f1}
			G&\leq C\|u(t)\|_{B^s_{p,r}}\Big[\Big(1+\|u(t)\|_{L^{\infty}}\Big)^3+\|u_x(t)\|_{L^{\infty}}\Big]\notag\\
			&\leq C\|u(t)\|_{B^s_{p,r}}\|u_x(t)\|_{L^{\infty}}.
		\end{align}
		Plugging \eqref{f1} into \eqref{f0}, it follows that
		\begin{align}\label{f2}
			\|u(t)\|_{B^s_{p,r}}+\lambda 	\|u(t)\|_{L^1_t([0,T];B^s_{p,r})}\leq \|u_0\|_{B^s_{p,r}}+C\int_{0}^t\|u\|_{B^s_{p,r}}\Big[\Big(1+\|u\|_{L^{\infty}}\Big)^3+\|u_x\|_{L^{\infty}}\Big](t')dt',
		\end{align}
		which implies
		\begin{align}\label{1b}
			\|u(t)\|_{B^s_{p,r}}\leq \|u_0\|_{B^s_{p,r}}\exp{C\int_0^t\|u_x(t')\|_{L^{\infty}}dt'}.
		\end{align}
		If $T^*<\infty,$ and $\int_0^{T^*}\|u_x\|_{L^{\infty}}(t')dt'<\infty,$ we know that $u\in L^{\infty}([0,T^*);B^s_{p,r}).$ This contradicts the definition that $T^*.$
		
		According to Theorem \ref{local well-pose} and the fact that $B^s_{p,r}\hookrightarrow L^{\infty},$ if $\int_0^{T^*}\|u_x(t')\|_{L^{\infty}}dt'=\infty,$ then $u$ must blow up in finite time.
	\end{proof}
	\begin{lemm}\cite{Guozi2019ill}\label{log}
		Let $u\in H^2$ be a solution to \eqref{001}. Then, we have
		$$\|u_x\|_{L^{\infty}}\leq C\Big(\|u_x\|_{B^{0}_{{\infty},{\infty}}}\cdot \ln(2+\|u\|_{H^2})+1\Big).$$
	\end{lemm}
	Next we give another blow-up criterion for  \eqref{001} to claim norm inflation in critical Besov spaces.
	\begin{lemm}\label{blowup}
		Let $u_0\in H^2,$ and let $T^*$ be the maximal existence time of the corresponding  solution $u$ to \eqref{001}. Then $u$ blows up in finite time $T^*<\infty$ if and only if
		$$\int_0^{T^*}\|u_x(t')\|_{B^0_{{\infty},\infty}}dt'=\infty.$$
	\end{lemm}
	\begin{proof}
		Combining Lemma \ref{R} with $\|u(t)\|_{H^1}\leq \|u_0\|_{H^1}$ for $t\in [0,T^*),$  we get
		\begin{align}\label{42}
			~~~&~\|u(t)\|_{H^2}+\lambda 	\|u(t)\|_{L^1_t(H^2)}\notag\\
			&\leq \|u_0\|_{H^2}+C\int_0^t\Big(\|u_x\|_{L^{\infty}}\|u\|_{H^1}+\|u_x\|_{H^1}\|u_x\|_{B^{0}_{\infty,\infty}}+ \|u\|_{H^2}+\|Q\|_{H^2}\Big)(t')dt'\notag\\
			&\leq \|u_0\|_{H^2}+C\int_0^t\|u\|_{H^2}\Big(1+\|u_x\|_{L^{\infty}}\Big)(t')dt'.
		\end{align}
		Using Lemma \ref{log} , we have
		\begin{align}\label{45}
			\|u\|_{H^2}\leq \|u_0\|_{H^2}+C\int_0^t\Big(1+\|u_x\|_{B^0_{\infty,\infty}}\ln(e+\|u\|_{H^2})\|u\|_{H^2}\Big)dt'.
		\end{align}
		The Gronwall inequality entails that
		\begin{align*}
			\|u\|_{H^2}\leq \|u_0\|_{H^2}e^{Ct+\int_0^t\|u_x\|_{B^0_{\infty,\infty}}\ln(e+\|u\|_{H^2})dt'}.
		\end{align*}
		Hence, it follows that
		$$\ln(e+\|u\|_{H^2})\leq \Big(\ln(e+\|u_0\|_{H^2})+Ct\Big)e^{C\int_0^t\|u_x\|_{B^{0}_{\infty,\infty}}dt'}.$$
		If $T^*<\infty,$ and $\int_0^{T^*}\|u_x(t')\|_{B^{0}_{\infty,\infty}}dt'<\infty,$ we deduce that $u\in L^{\infty}([0,T^*];H^2),$ which contradicts the assumption that  $T^*$ is the maximal existence time.
		On the other hand, since the fact that  $L^{\infty}\hookrightarrow B^0_{\infty,\infty},$   we get  $u$  must blow up in  finite time whenever $\int_0^{T^*}\|u_x(t')\|_{B^0_{\infty,\infty}}dt'=\infty.$
	\end{proof}
	\subsection{Blow-up}
	\par
	The following lemma shows that if the initial value satisfies certain conditions, the strong solution of  \eqref{001} will blow up in finite time.
	\begin{theo}\cite{Igor02decay}\label{blow up u}
		Given the initial datum $u_0\in H^3(\mathbb{R}),$ let $u=u(t,x)$ be the corresponding solution of \eqref{001},  $f(t)=\sup\limits_{x\in \mathbb{R}},$ and $\kappa: =max\{|\alpha|, \frac{|\beta|}{3},\frac{|\gamma|}{4},|\Gamma|\}.$ There exists a $\eta\in (0,\eta_0]$ and $x_0\in \mathbb{R}$ such that $\eta u_{x0}(x_0)<min\{-\|u_0\|^{\frac 1 2}_{H^1}, -\|u_0\|^{2}_{H^1}\}$ with $\eta_0=\sqrt{\frac{2}{1+12\kappa}}.$
		
		If
		\begin{align*}
			\lambda\in\Big(0,-\frac{f(0)} {4} \frac{\eta^2u^2_{x0}(x_0)-max\{-\|u_0\|_{H^1}, -\|u_0\|^{4}_{H^1}\}}{\eta^2u^2_{x0}(x_0)}\Big),
		\end{align*}
		then the solution of  \eqref{001} blows up in finite time.
	\end{theo}
	Now we give another blow-up result.
	\begin{theo}\label{blow-up n}
		Let $u_0\in H^s(\mathbb{R})$ with $s>\frac 3 2.$ Assume $u(t,x)$ is the corresponding solution of  \eqref{001} with the initial value $u_0.$ For $n\in \mathbb{N} / \{0\},$ if the slope of $u_0$ satisfies
		$$f(0)\triangleq \int_{\mathbb{R}} u^{2n+1}_{0x}dx<-\Big(\frac{8cK^2\|u_0\|^{\frac{2}{2n-1}}_{H^1}}{2n-1}\Big)^{\frac{2n-1}{2n}},$$
		with
		\begin{align*}
			K^2&=\frac{\lambda (n+1)\Big(4c\lambda (n+1)\|u_0\|^{\frac{2}{2n-1}}_{H^1}\Big)^{2n-1}}{2n^{2n}}+\frac{(2n+1)(C_1+C_3)}{n+1}\Big(\frac{8n(2n+1)C_2}{(2n-1)(n+1)}\Big)^n,\\
			C_1+C_3&=\frac{2(|\alpha|+|\Gamma|)\|u_0\|^{n+1}_{H^1}}{({\sqrt{2}})^{n-1}}+\frac{|\beta|(\|u_0\|^{3(n+1)}_{H^1}+\|u_0\|^{3n+5}_{H^1})}{3({\sqrt{2}})^{3n+1}}+\frac{|\gamma|\|u_0\|^{4n+4}_{H^1}}{2^{2(n+1)}}+\frac{\|u_0\|^{2(n+1)}_{H^1}}{2^n},\\
			C_2&=|\alpha|+|\Gamma|+\frac{|\beta|}{3}+\frac{|\gamma|}{4}+\frac 1 2 ,
		\end{align*}
	and $c$ is  the Gagliardo-Nirenberg constant, then there exists  a lifespan $T<\infty$ such that the solution $u$ blows up in finite time $T$. Moreover, the above bound of lifespan $T$ satisfies
		\begin{align}\label{n1}
			T\leq \int_{-f(0)}^{\infty}\frac{dy}{\frac{2n-1}{4c}\|u_0\|^{-\frac{2}{2n-1}}_{H^1}y^{\frac{2n}{2n-1}}-K^2},
		\end{align}
		where $n\in \mathbb{N} / \{0\},$  $c$ is a universal  Gagliardo-Nirenberg constant.
	\end{theo}
	\begin{proof}
		Letting the function $p=\frac{1}{2} e^{-|x|},$ then $p\ast f=\Lambda^{-2}f$ for any $f\in L^2(\mathbb{R}).$ Differentiating \eqref{001} with respect to  variable $x$,  we have \begin{align}\label{n2}
			u_{tx}+(u+\Gamma)u_{xx}+\lambda u_x=-h(u)+u^2-\frac 1 2 u_x^2+P.
		\end{align}
		In view of  \eqref{n2}, for any $n\in \mathbb{N} {/} \{0\},$ we infer that
		\begin{align}\label{n3}
			\frac{d}{dt}\int_{\mathbb{R}}u^{2n+1}_xdx+\lambda(2n+1) \int_{\mathbb{R}}u^{2n+1}_xdx &=-\frac{2n-1}{2}\int_{\mathbb{R}}u^{2n+2}_xdx\underbrace{-(2n+1)\int_{\mathbb{R}}h(u)u^{2n}_xdx}_{\mathcal{I}_1}\notag\\
			&~~~+\underbrace{(2n+1)\int_{\mathbb{R}}u^2u^{2n}_xdx}_{\mathcal{I}_2}\underbrace{-(2n+1)\int_{\mathbb{R}}P\cdot u^{2n}_xdx}_{\mathcal{I}_3},
		\end{align}
		with $P=\Lambda^{-2} (-h(u)+u^2+\frac 1 2u_x^2).$ \\
		Using Lemma \ref{conser} and H{\"o}lder's inequality yield that
		\begin{align}\label{n4}
			\int_{\mathbb{R}}\underbrace{(p\ast u)}_{P_1} u^{2n}_xdx&\leq \Big(\int_{\mathbb{R}}P_1^{n+1}dx\Big)^{\frac{1}{n+1}}\Big(\int_{\mathbb{R}}u^{2n+2}_xdx\Big)^{\frac{n}{n+1}}\\ \notag&\leq \frac{\|P_1\|^{n-1}_{L^{\infty}}\cdot\|P_1\|^2_{L^2}}{(n+1){\epsilon}^{n+1}}+\frac{n\epsilon^{\frac{n+1}{n}}}{n+1}\int_{\mathbb{R}}u^{2n+2}_xdx \notag\\
			&\leq \frac{\|u_0\|^{n+1}_{H^1}}{({\sqrt{2}})^{n-1}(n+1){\epsilon}^{n+1}}+\frac{n\epsilon^{\frac{n+1}{n}}}{n+1}\int_{\mathbb{R}}u^{2n+2}_xdx,
		\end{align}
		where we use the fact that $\|P_1\|_{L^{\infty}}\leq \|u\|_{L^{\infty}}\leq  \frac{\|u_0\|_{H^1}}{\sqrt{2}}$ and $\|P_1\|_{L^{2}}\leq  \|u\|_{L^2}\leq \|u_0\|_{H^1}.$ Note that $\int_0^t p(x)\ast (u^2+\frac 1 2 u_x^2)\cdot u_x^{2n}dx>0.$ Similar to the above inequality, we deduce that 
		\begin{align}\label{n99}
			\sum_{i=1}^{3}\mathcal{I}_i&\leq \frac{2n+1}{(n+1){\epsilon}^{n+1}}\Big[\underbrace{\frac{2(|\alpha|+|\Gamma|)\|u_0\|^{n+1}_{H^1}}{({\sqrt{2}})^{n-1}}+\frac{|\beta|(\|u_0\|^{3(n+1)}_{H^1}+\|u_0\|^{3n+5}_{H^1})}{3({\sqrt{2}})^{3n+1}}}_{C_1}\Big]\notag\\&~~+\frac{(4n^2+2n)\epsilon^{\frac{n+1}{n}}}{n+1}\underbrace{(|\alpha|+|\Gamma|+\frac{|\beta|}{3}+\frac{|\gamma|}{4}+\frac 1 2)}_{C_2}\int_{\mathbb{R}}u^{2n+2}_xdx
			\notag\\&~~+\frac{2n+1}{(n+1){\epsilon}^{n+1}}\Big[\underbrace{\frac{|\gamma|\|u_0\|^{4n+4}_{H^1}}{2^{2(n+1)}}+\frac{\|u_0\|^{2(n+1)}_{H^1}}{2^n}}_{C_3}\Big].
		\end{align}
		Choosing  $\epsilon=\Big(\frac{(2n-1)(n+1)}{8n(2n+1)C_2}\Big)^{\frac{n}{n+1}},$ and defining $f(t)=\int_{\mathbb{R}}u^{2n+1}_xdx,$ we obtain 
		\begin{align}\label{n6}
			\frac{d}{dt}f(t)\leq-\frac{2n-1}{4}\int_{\mathbb{R}}u^{2n+2}_xdx+\lambda(2n+1) |f(t)|+K_1^2,
		\end{align}
		with $K^2_1=\frac{(2n+1)(C_1+C_3)}{n+1}\Big(\frac{8n(2n+1)C_2}{(2n-1)(n+1)}\Big)^n.$

		Combining the Gagliardo-Nirenberg inequality with   the Young inequality, it follows that
		\begin{align}\label{n8}
			\Big(\int_{\mathbb{R}}u_x^{2n+1}dx\Big)^{\frac{2n}{2n-1}}&\leq
			c\Big(\int_{\mathbb{R}}u^2_xdx\Big)^{\frac{1}{2n-1}}\int_{\mathbb{R}}u_x^{2n+2}dx\leq c\|u_0\|^{\frac{2}{2n-1}}_{H^1}\int_{\mathbb{R}}u_x^{2n+2}dx,
		\end{align}
		and
		\begin{align}\label{n10}
			\lambda (2n+1)|f(t)|\leq \frac{1}{2n}\Big(\frac{\lambda (2n+1)}{\epsilon_0} \Big)^{2n}+\frac{2n-1}{2n}\Big({|f(t)|}{\epsilon_0}\Big)^{\frac{2n}{2n-1}}.
		\end{align}	
		Let $\epsilon_0 =\Big(\frac{n} {4c\|u_0\|^{\frac{2}{2n-1}}_{H^1}}\Big)^{\frac{2n-1}{2n}}.$ After a few calculations, we have
		\begin{align}\label{n12}
			\frac{d}{dt}f(t)\leq -\frac{2n-1}{8c\|u_0\|^{\frac{2}{2n-1}}_{H^1}}f^{\frac{2n}{2n-1}}(t)+K^2,
		\end{align}
		where $K^2=K^2_1+\frac{\lambda (n+1)\Big(4c\lambda (n+1)\|u_0\|^{\frac{2}{2n-1}}_{H^1}\Big)^{2n-1}}{2n^{2n}}.$
		Using the fact that
		$f(0)<-\Big(\frac{8cK^2\|u_0\|^{\frac{2}{2n-1}}_{H^1}}{2n-1}\Big)^{\frac{2n-1}{2n}},$
		we then get $f'(t)<0$ and $f(t)$ is a decreasing function. Suppose that $u$ is a global solution,  there exists a  time $t_1>0$ such that
		\begin{align}\label{n13}
			f'(t)\leq -\frac{2n-1}{16c\|u_0\|^{\frac{2}{2n-1}}_{H^1}}f^{\frac{2n}{2n-1}}(t),~~t\geq t_1,
		\end{align}
		which implies that
		\begin{align}\label{n14}
			\frac{1}{16c\|u_0\|^{\frac{2}{2n-1}}_{H^1}}(t-t_1)+f^{\frac{-1}{2n-1}}(t_1)\leq f^{\frac{-1}{2n-1}}(t)\leq 0, ~~t\geq t_1.
		\end{align}
		Thanks to $f(t_1)<0,$ if $t\geq t_1$ large enough, we deduce that the above inequality is not vaild. Therefore, there exists a lifespan $T<\infty$ such that $\lim_{t\rightarrow T} f(t)=-\infty.$  By solving \eqref{n12}, we end up with 
		\begin{align*}
			T\leq -\int_0^{T}\frac{f'(t)dt}{\frac{2n-1}{8c}\|u_0\|^{-\frac{2}{2n-1}}_{H^1}f^{\frac{2n}{2n-1}}(t)-K^2},
		\end{align*}
		which menas that
		\begin{align}\label{n15}
			T\leq\int_{f(0)}^{\infty}\frac{dy}{\frac{2n-1}{8c}\|u_0\|^{-\frac{2}{2n-1}}_{H^1}y^{\frac{2n}{2n-1}}-K^2}.
		\end{align}
		Then Lemma \ref{blow up u} guarantees that
		$$\lim\inf_{t\rightarrow T}\Big(\inf_{x\in \mathbb{R}} u_x(t,x)\Big)=-\infty.$$
	\end{proof}
	\subsection{Ill-posedness}
	\par
	\begin{theo}\label{ill}
		Let $1\leq p\leq \infty$ and $1< r\leq \infty.$ For any $\epsilon>0,$ there exists $u_0\in H^{\infty}$ such that the following holds\\
			$(1)$$~\|u_0\|_{B^{1+\frac 1 p}_{p,r}}\leq \epsilon;$\\
			$(2)~$There is a unique solution $u\in C([0,T];H^{\infty})$ to the equation \eqref{001} with a maximal lifespan
			$T<\epsilon;$\\
			$(3)~$ $\lim\sup\limits_{t\rightarrow T^-}\|u\|_{B^{1+\frac 1 p}_{p,r}}\geq \lim\limits_{t\rightarrow T^-}\|u\|_{B^1_{\infty,\infty}}=\infty.$	
	\end{theo}
	\begin{proof}
		Fix $1\leq p\leq \infty$ and $1<r\leq \infty,$ and $\epsilon>0.$ Let
		\begin{align}\label{416}
			g(x)=\sum_{j\geq 1} \frac{1}{2^j j^{\frac{2}{1+r}}}g_j(x),
		\end{align}
		with  $\hat{g}_j(\xi)=i2^{-j}\xi \tilde{\chi}(2^{-j}\xi).$ Let $\chi$ be a non-negative, non-zero $C_0^{\infty}$ function satisfies $\tilde{\chi}\chi_0=\tilde{\chi}.$ Hence,  we obtain that $\Delta_j g(x)=\frac{1}{2^j j^{\frac{2}{1+r}}}g_j(x),$  $\|\Delta_j g(x)\|_{L^p}\sim  \frac{2^{\frac {j}{p'}}}{2^j j^{\frac{2}{1+r}}},$ and
		$$\|g\|_{B^{1+\frac 1 p}_{p,q}}\sim \|\frac{1}{j^{\frac{2}{1+r}}}\|_{l^q},$$
		from which it follows $g\in B^{1+\frac 1 p}_{p,r}\setminus B^{1+\frac 1 p}_{p,1},$ and
		$$g'(0)=\int \hat{g}'(\xi)d\xi=\int 2\pi i\xi\hat{g}(\xi)d\xi=-c\sum_{j\geq 1} \frac{1}{j^{\frac{2}{1+r}}}=-\infty.$$
		For any $\epsilon >0,$ let $u_{0,\epsilon}=\|g\|^{-1}_{B^{1+\frac 1 p}_{p,r}}\cdot \epsilon S_{K}(g)$ where $K$ is a large enough such that $\eta u'_{0,\epsilon}(0)<\min\{-\|u_0\|^{\frac 1 2}_{H^1},$\\ $\|u_0\|^{2}_{H^1}\}$. Therefore, $u_{0,\epsilon}\in H^{\infty}, \|u_{0,\epsilon}\|_{B^{1+\frac 1 p}_{p,r}}\leq \epsilon.$ Taking advantage of  Theorem \ref{blow up u} , we get there exists an unique associated solution $u\in C([0,T];H^{\infty})$ with maximal  lifespan $T<\epsilon.$ By Lemmas \ref{log}-\ref{blowup} , we can prove that $\lim\sup_{t\rightarrow T^-}\|u\|_{B^{1}_{\infty,\infty}}=\infty.$
	\end{proof}
	\smallskip
	\noindent\textbf{Acknowledgments}
	This work was partially supported by NNSFC (Grant No. 12171493 ),  FDCT (Grant No. 0091/2018/A3), the Guangdong Special Support Program (Grant No.8-2015).

\end{document}